\documentclass[10pt]{amsart} 
\usepackage{amsfonts,amssymb,amscd,amsmath,enumerate,verbatim,calc} 
 
\newtheorem{theorem}{Theorem}[section]
\newtheorem{lemma}[theorem]{Lemma}
\theoremstyle{definition}
\newtheorem{definition}[theorem]{Definition}

\theoremstyle{remark}
\newtheorem{remark}[theorem]{Remark}
\numberwithin{equation}{section}
\theoremstyle{plain}

\def\C{\mathbb C}
\def\I{\rm I}
\def\F{\mathbb F}
\def\H{\mathbb H}

\def\V{\mathbb V}
\def\W{\mathbb W}

\def\E{\mathbb E}
\def\F{\mathbb F}
\def\H{\mathbb H}

\def\R{\mathbb R}

\def\V{\mathbb V}
\def\W{\mathbb W}

\def\K{\mathbb K}
\copyrightinfo{2001}{enter name of copyright holder}

\newcommand{\thmref}[1]{Theorem~\ref{#1}}
\newcommand{\lemref}[1]{Lemma~\ref{#1}}
\newcommand{\remref}[1]{Remark~\ref{#1}}

\newcommand{\eqnref}[1]{~{\textrm(\ref{#1})}}
\newcommand{\defref}[1]{Definition~\ref{#1}}

\begin{document}
 
\title{The $z$-Classes of Isometries}
\author{Krishnendu Gongopadhyay \and Ravi S. Kulkarni}
\address{Department of Mathematical Sciences, Indian Institute of Science Education and Research (IISER) Mohali, Knowledge city, Sector 81, S.A.S. Nagar, P.O. Manauli 140306, India}
\thanks{Gongopadhyay acknowledges the support of SERC-DST FAST  grant {SR/FTP/MS-004/2010}. }
\email{krishnendu@iisermohali.ac.in,  krishnendug@gmail.com}
\address{Bhaskaracharya Pratishthana,  56/14, Erandavane, Damle Path, 
Off Law College Road, Pune  411 004, India}
\email{punekulk@gmail.com}
\date{\today}
\begin{abstract}
Let $G$ be a group. Two elements $x,y$ are said to be in the same \emph{z-class} if their centralizers are conjugate in $G$.  Let $\V$ be a vector space of dimension $n$ over a field $\F$ of characteristic different from $2$. Let $B$ be a non-degenerate symmetric, or skew-symmetric, bilinear form on $\V$. Let ${\rm I}(\V, B)$ denote the group of isometries of $(\V, B)$. We show that the number of $z$-classes in ${\rm I}(\V, B)$ is finite when $\F$ is perfect and has the property that it has only finitely many field extensions of degree at most $n$.
\end{abstract}
\maketitle
%\tableofcontents
\section{Introduction}\label{intro}

Let $G$ be a group. We define an equivalence relation $\sim$ on $G$ as follows:  for $x$, $y$ in $G$, $x \sim y$ if the centralizers $Z_G(x)$ and $Z_G(y)$ are conjugate in $G$.  The equivalence class of $x$ is called the \emph{$z$-class} of $x$ in $G$. The $z$-classes are pairwise disjoint and give a partition of the group $G$. This provides important information about the internal structure of the group, see \cite{rkrjm} for further details. The structure of each $z$-class can be expressed as a certain set theoretic fibration, see \cite[Theorem 2.1 ]{rkrjm}. In general, a group may be infinite and it may have infinitely many conjugacy classes, but the number of $z$-classes is often finite. For example, if $G$ is a compact Lie group, then it is implicit in Weyl's structure theory see, \cite{weyl}, Borel-de Siebenthal \cite{bds}, that the number of $z$-classes in $G$ is finite. Analogously, Steinberg  \cite[p.107]{stein} has remarked on the finiteness of $z$-classes in  reductive algebraic groups over an algebraically closed field of \emph{good} characteristic. In \cite{rkrjm}, Kulkarni proposed to interpret the $z$-classes as an internal ingredient in a group $G$ that can be used to make precise the intuitive notion of  ``dynamical types" in the \hbox{$G$-action} on any set $X$. The Fibration Theorem, see \cite[Theorem 2.1]{rkrjm},  gives a set-theoretic fibration of the $z$-class of $x$ with base the homogeneous space $G/N(x)$, where $N(x)$ is the normalizer of $Z_G(x)$ in $G$, and a fiber consists of the elements $y$ in the center of $Z_G(x)$ such that $Z_G(x)=Z_G(y)$. 
For example, in classical  geometries over $\R$, $\C$ or $\H$, it is observed that the ``dynamical types" that our mind can perceive are just finite in number and  this finiteness of ``dynamical types'' can be interpreted as a phenomenon related to the finiteness of the $z$-classes in the corresponding group of the geometry.  With this motivation, the $z$-classes in the isometry group of the $n$-dimensional real hyperbolic space were classified and counted  in \cite{gk1}. It is also an interesting problem to classify the $z$-classes in other linear groups that appear as isometry group in rank one symmetric spaces of non-compact type. The $z$-classes in the isometry group ${\rm Sp}(n,1)$ of the $n$-dimensional quternionic hyperbolic space have been classified and counted in \cite{gjg}.   Classification of the $z$-classes in ${\rm U}(n,1)$, the isometry group of the $n$-dimensional complex hyperbolic space,  has been obtained in \cite{cg2}, also see \cite[Appendix]{gjg}. Recently, $z$-classes have also been used in the context of classifying the isometries in hyperbolic geometries, see,  \cite{cg, go1, go2}. 

 In addition to these, it is of independent algebraic interest to parametrize both the conjugacy and the $z$-classes in a group.  For example, the problem can be asked for finite groups of Lie type; classical groups or exceptional groups.  The conjugacy classes, $z$-classes and the set of operators themselves of the general linear groups and the affine groups have been parametrized by Kulkarni \cite{kulkarni}. This has been extended to linear operators over division rings by  Gouraige \cite{rony}. In an attempt to understand the $z$-classes in exceptional groups,  Singh \cite{as} has proved a finiteness result for the $z$-classes in the compact real form $G_2$.

\medskip  Let $\F$ be a field of characteristic different from 2. Let $\V$ be a vector space of dimension $n$ over $\F$.  Let   $\V$ be equipped with a non-degenerate symmetric or  skew-symmetric bilinear form $B$.  The group of isometries of $(\V, B)$ is denoted by ${\rm I}(\V, B; \F)$, or simply ${\rm I}(\V, B)$ when the underlying field is fixed. When $B$ is symmetric, resp. skew-symmetric, ${\rm I}(\V, B)$ is the orthogonal, resp. symplectic group. 
 In this paper we ask for the $z$-classes in ${\rm I}(\V, B)$. Our main theorem is the following. 

\begin{theorem}\label{zc}
If $\F$ is perfect and has the property that it has only finitely many field extensions of degree at most $\dim \V$, then the number of $z$-classes in ${\rm I}(\V, B)$ is finite. This holds for example when the field $\F$ is algebraically closed, the field of real numbers, or a local field.  
\end{theorem}
Along the way we parametrize the $z$-classes of the semisimple elements, see,  \hbox{\thmref{zcs}}. 

A first step in understanding of the $z$-classes is to classify the conjugacy classes. There has been a considerable amount of work on the conjugacy problem in orthogonal and symplectic groups, see,  Asai \cite{asai}, \hbox{Burgoyne} and Cushman \cite{bc}, Kiehm \cite{k}, Milnor \cite{milnor},  Springer-Steinberg \cite{ss}, Wall \cite{wall} and Williamson \cite{will}. A common theme of these works is to reduce the conjugacy problem to the equivalence problem for Hermitian forms. Our conjugacy classification has a similar flavor. However,  a notable \hbox{feature} of our classification is that,  in the ``generic" case when the minimal polynomial of an element in ${\rm I}(\V, B)$  is a prime-power,   it gives an explicit parametrization  of the conjugacy classes, see, \thmref{ppp}. Consequently, we also obtain a parametrization of the $z$-classes in this case, see, \thmref{zc1}.  As we shall see,  the $z$-classification  depends on  the equivalence problem of Hermitian forms over arbitrary fields.  The equivalence problem of   Hermitian forms was solved by Kiehm \cite{k} and Wall \cite{wall}. We will not get into the equivalence problem of the Hermitian spaces in this paper. However, our classification of the $z$-classes is enough to prove our main result, \thmref{zc}.

\section{Preliminaries}\label{prel}
\subsection{Self-dual polynomial}\label{reciprocal}
  Let $\F[x]$ be the ring of polynomials over $\F$. For a polynomial $g(x)$ let $c_k(g)$ denote the coefficient of $x^k$ in $g(x)$. Let $\bar \F$ denote the algebraic closure of $\F$. 
 
Let $f(x)$ be a monic polynomial of degree $n$ over $\F$ such that $0, 1$ and $-1$ are not its roots. Over $\bar \F$ let
$$f(x)=(x-c_1)(x-c_2)....(x-c_n).$$
Then the polynomial 
$$f^{\ast}(x)=(x-c_1^{-1})(x-c_2^{-1})....(x-c_n^{-1})$$
is said to be the \emph{dual to $f(x)$}. It is easy to see that
$$f^{\ast}(x)=f(0)^{-1} x^n f(x^{-1}).$$  
Clearly, $c_k(f^{\ast})=f(0)^{-1}c_{n-k}$. 

\begin{definition}
Let $f(x)$ be a monic polynomial over $\F$ such that $-1$, $0$, $1$ are not its roots. The polynomial $f(x)$ is called \emph{reciprocal}, or \emph{self-dual}, if $f(x)=f^{\ast}(x)$. 
\end{definition}
Thus if $f(x)$ self-dual, then the degree $n$ of $f(x)$ is even, and for all $k$, $c_k(f)=c_{n-k}(f)$.

\subsection{Decomposition of the space relative to an isometry}\label{primary}

 Suppose { $T: \V \to \V$}  is an element in ${\rm I}(\V, B)$. Let $m_T(x)$ denote the minimal polynomial of $T$. Suppose $p_1(x), \ldots ,p_l(x)$ are irreducible polynomials over $\F$ such that $m_T(x)= p_1(x)^{d_1}\ldots p_l(x)^{d_l}$, where for $i \neq j$, $p_i(x) \neq p_j(x)$. Suppose degree of $m_T(x)$ is $m$. The integer $d_i$ is called the \emph{exponent}, or the \emph{multiplicity}, of the prime factor $p_i(x)$.

Let $\E=\F[x]/(m_T(x))$. The image of the indeterminate $x$ in $\E$ is denoted by $t$. There is a canonical algebra structure on $\E$ defined by $tv=Tv$.  The $\F$-algebra $\E=\F[t]$ is  spanned by $\{1, t, t^2,\cdots,t^{m-1}\}$. In particular, if the minimal polynomial is irreducible, then $\E$ is an extension field of $\F$. The following lemma follows from Lemma 4.1 in \cite{gk2}.

\begin{lemma}\label{aut}
(i) The minimal polynomial of an element $T$ in ${\I}(\V, B)$  is self-dual.

(ii)There is a unique automorphism $e \to \bar e$ of $\E$ over $\F$ which carries $t$ to $t^{-1}$.
\end{lemma}
Thus an irreducible factor $p(x)$ of the minimal polynomial can be one of the following three types:
 
 (i) $p(x)$ is self-dual.
 
 (ii) $p(x)=x-1$, or, $x+1$.
 
 (iii) $p(x)$ is not self-dual. In this case there is an irreducible factor $p^{\ast}(x)$ of the minimal polynomial such that $p^{\ast}(x)$ is dual to $p(x)$.

Among the irreducible factors of $m_T(x)$, suppose $p_i(x)$ is self-dual for $i=1,2,...,k_1$. Let the other irreducible factors be $p_j(x), p_j^{\ast}(x)$ for $j=1,2,...,k_2$ with $p_j(x) \neq p_j^{\ast}(x)$. For a prime-power  polynomial $p(x)^d$,  let $\V_p=\;ker\; p(T)^d$. Let $\oplus$ denote the orthogonal sum, and $+$ denote the usual sum of subspaces. It can be seen easily that there is a \emph{primary decomposition} of $\V$ (with respect to $T$) into $T$-invariant non-degenerate subspaces:
\begin{equation}\label{decom}
\V=\oplus_{i=1}^{k_1} \V_i \bigoplus \oplus_{j=1}^{k_2} \V_j
\end{equation}
where for $i=1,2,...,k_1$, $p_i(x)$ is self-dual, $\V_i=\V_{p_i}$, and $B|_{\V_i}$ is non-degenerate; for $j=1,2,...,k_2$, $\V_j=\V_{p_j} + \V_{p^{\ast}_j}$, $B|_{\V_{p_j}}=0=B|_{\V_{p_j^{\ast}}}$, here $p_j(x) \neq p_j^{\ast}(x)$.  Let  $T_l$ denote the restriction of $T$ to $\V_l$. Then $m_{T_i}(x)=p_i(x)^{d_i}$ for $i=1,2,...,k_1$, and $m_{T_j}(x)=p_j(x)^{d_j}p_j^{\ast}(x)^{d_j}$ for $j=1,2,...,k_2$.  Let $Z(T)$ denote the centralizer of $T$ in ${\I}(\V, B)$. We observe that the decomposition \eqnref{decom} is in fact invariant under $Z(T)$. Moreover we have a canonical decomposition
$$Z(T)=\Pi_{i=1}^{k_1}Z(T_i) \times \Pi_{j=1}^{k_2}Z(T_j).$$
Thus the conjugacy classes and the $z$-classes of $T$ are determined by the restriction of $T$ to each of the primary subspaces. Hence it is enough to determine the conjugacy and the $z$-classes of an isometry $T:\V \to \V$ with minimal polynomial $m_T(x)=p(x)^d$, where $p(x)$ is one of the types (i), (ii), (iii) above.

Finally note the following lemma. For a proof of the lemma, see, \cite[Lemma 4.2]{gk2}.
\begin{lemma}\label{pd2}
Let $T$ be an element in ${\I}(\V, B)$. Suppose $T:\V \to \V$ is such that the minimal polynomial  is one of the types (i), (ii) above. Suppose $m_T(x)=p(x)^d$.   There is an orthogonal decomposition  $\V=\oplus_{i=1}^k \V_{d_i}$,
 where $1 \leq d_1 <\cdots< d_k=d$, and for each $i=1,...,k$, $\V_{d_i}$ is free over the algebra $\F[x]/(p(x)^{d_i})$. For each $i$, the summand $\V_{d_i}$ corresponds to the elementary divisor $p(x)^{d_i}$ of $T$.
\end{lemma}

\begin{remark}\label{r0}
In the above lemma, suppose $\deg p(x)=m$. Then

\noindent  $\dim_{\F} \F[x]/(p(x)^{d_i})=md_i$. Suppose  $\V_{d_i}$ has dimension $l_i$ as a free module over $\F[x]/(p(x)^{d_i})$. Thus $\dim_{\F} \V_{d_i}=md_i l_i$. This gives us a secondary \hbox{partition} $\pi: \frac{n}{m}=\sum_{i=1}^k d_i l_i$. 
\end{remark}
We end this section with the following definition. 
\begin{definition}
Let $R$ be a commutative ring with involution $e \mapsto \bar e$. Let $\epsilon=1$ or $-1$. An $\epsilon$-Hermitian form on an $R$-module $M$ is a sesquilinear mapping $s: M \times M \to R$ such that for all $x, y \in M$, 
$$s(x, y)=\epsilon   \overline{s(y, x)}.$$
That is for $\epsilon=1$, $s$ is Hermitian; for $\epsilon=-1$, $s$ is skew-Hermitian. 
\end{definition}

\section{The induced form and the conjugacy classes}
\subsection{The minimal polynomial is prime-power}
\begin{lemma}\label{lss} {\tt(Springer-Steinberg \cite{ss})}
  Let $T: \V \to \V$  in ${\I}(\V, B)$ be  such that $m_T(x)=p(x)^d$, where $p(x)$ is an irreducible  polynomial over $\F$.  Assume that $p(x)$ is either self-dual, or, $x-1$. If  $p(x)=x-1$, then assume $d>1$.  Consider the cyclic $\F$-algebra $\E_d^T=\F[x]/(p(x)^d)$. We simply denote it by $\E^T$ when there is no confusion about $d$. The $\E^T$-module $\V$ is denoted by $\V^T$.
Then we have the following. 

(i) There is a unique automorphism $e \to \bar e$ of $\E^T$ over $\F$ which carries $t$ to $t^{-1}$. 

(ii) There exists an $\F$-linear function $h^T:\E^T \to \F$ such that the symmetric bilinear map 
$\bar h^T: (a, b) \mapsto h^T(ab)$ on $\E^T \times \E^T$ is non-degenerate. Also there exists $c \in \E^T$ such that for all $e \in \E^T$, $h^T(\bar e)=h^T(ce)$. Moreover,  if $p(x) \neq x-1$, we can take $c=1$. If $p(x)=x-1$, then $c=(-1)^{d-1}$.  
\end{lemma}

For a proof of the above lemma cf. Springer-Steinberg \cite[p.254]{ss}. For a proof when the field extension $\F_d=\F[x]/(p(x))$ is separable, cf. Asai \cite[p.329]{asai}. Applying the above lemma we have the following theorem. The theorem is implicit in the work of Springer-Steinberg \cite{ss}. Milnor \cite{milnor} gave a version of the following theorem when $T$ is semisimple. We have given a detailed proof in the general case. The proof is essentially imitating Milnor's proof in the semisimple case.

\begin{lemma}\label{tm}
 The module $\V$ over $\E^T$ admits a unique $\epsilon$-Hermitian form  $H^T(u,v)=\epsilon \overline{H^T(v,u)}$, $\E^T$-linear in the first variable, and is related to the original $\F$-valued inner product by the identity
\begin{equation} \label{hf} B(u, v)=h^T(H^T(u, v)).
\end{equation}
\end{lemma}
\begin{proof}
For $u$, $v$ in $\V$, consider the linear map $L:\E^T \to \F$ given by
$L(e)=B(eu, v)$. There exists a unique $e'$ in $\E^T$ such that
$h^T(ee')=L(e)$.  We define $H^T(u,v)$ to be this element $e'$. That is, $H^T(u, v)$ is defined as follows: 
$$\hbox{for all } e \hbox{ in } \E^T, \hbox{ and for } u, v \hbox{ in } \V, \;\;h^T(eH^T(u, v))=B(eu, v).$$ 
In particular taking $e=1$ we have
\begin{equation*}
 h^T(H^T(u, v))=B(u, v)
\end{equation*}
Now we see that for $u_1$, $u_2$, $v$ in $\V$, 
\begin{eqnarray*}
 h^T(e(H^T(u_1,v)+H^T(u_2,v))) & =& h^T(eH^T(u_1,v))+h^T(eH^T(u_2,v)) \\
&=& B(eu_1,v)+B(eu_2,v)\\
 &=& B(eu_1+eu_2, v) \\
&=&B(e(u_1+u_2), v)=h^T(H^T(u_1+u_2, v))
\end{eqnarray*}
\begin{equation}\label{ed1}
 \;\;\Rightarrow\; H^T(u_1,v)+H^T(u_2,v)=H^T(u_1+u_2,v) \\
\end{equation}
Now for all $e'$ in $\E^T$ we have
\begin{eqnarray*}
 h^T(e'eH^T(u,v)) &=&B(e'eu, v)\\
&=&B(e'(eu), v)=h^T(e'H^T(eu, v))
\end{eqnarray*}
\begin{equation}\label{ed2}
 \Rightarrow eH^T(u,v)=H^T(eu,v)
\end{equation}

This shows that $h^T$ is $\E^T$-linear in the first variable. 

Given any Hermitian form $H(u,v)$ satisfying \eqnref{hf} we see that 
$$h^T(eH(u,v))=h^T(H(eu,v))=B(eu, v).$$
Therefore $H^T(u,v)$ is unique.
 
Further, for all $e$ in $\E^T$,
\begin{eqnarray*}
 h^T(e(\overline{H^T(u,v)}))&=& \epsilon h^T(\bar{e}H^T(u,v)), \  \ \hbox{using part (ii) of \lemref{lss}}\\
 &=&\epsilon B(\bar e u, v) \\
 &=&\epsilon B(ev, u)=h^T(e\epsilon H^T(v,u))
\end{eqnarray*}
\begin{equation}
 \Rightarrow \overline{H^T(u,v)}=\epsilon H^T(v,u)
\end{equation}

This proves the theorem. 
\end{proof}

\begin{remark} \label{r1} Let $S:\V \to \V$ and $T: \V \to \V$ be two isometries such that 
$m_{S}(x)=p(x)^d$, $m_{T}(x)=q(x)^d$, where $p(x), \ q(x)$ are irreducible and self-dual, $\deg p(x)=\deg q(x)$ and $\E^S$ and $\E^T$ are $\F$-isomorphic. Let $s$ and $t$ are images of $S$ and $T$ in $\E^S$ and $\E^T$ respectively. Let  \hbox{$f:\E^S \to \E^T$}  be an $\F$-isomorphism such that $f(s)=t$.  Let $h^S:\E^S \to \F$ be the linear map as in \lemref{lss}. Then $h^T=h^S\circ f^{-1}$ is such a linear map on $\E^T$, and  this map induces a Hermitian form $H'$ on $\V^T$. Since such a Hermitian form is unique, hence we must have $H'=H^T$. Thus for $u$, $v$ in $\V^S$,  $h^S(H^S(u, v))=B(u, v)$, and for $u', v'$ in $\V^T$,
$h^T(H^T(u', v'))=h^S\circ f^{-1} (H^T(u', v'))$. \end{remark}

\begin{definition}\label{here}
Suppose $\E$ and $\E'$ are isomorphic modules over $\F$, and let $f: \E \to \E'$ be an isomorphism. Let $H$ be an $\E$-valued Hermitian form on $\V$ and let $H'$ be an $\E'$-valued Hermitian form on $\V'$. Then $(\V,H)$ and $(\V',H')$ are equivalent if there exists an $\F$-isomorphism $T:\V \to \V'$ such that for all $u, v$ in $\V$ and for all $e$ in $\E$ the following conditions are satisfied.

(i) $T(e v)=f(e)T(v)$, and

(ii) $H'(T(u), T(v))=f(H(u, v))$.

When $\E=\E'$, we take $f$ to be the identity in the definition. 
\end{definition}

\begin{theorem}\label{lppp}
Suppose $S$ and $T$ are isometries of $(\V, B)$. Let the minimal polynomial of both $S$ and $T$ be $(x-1)^d$ or,  $p(x)^d$,  where $p(x)$ is monic, self-dual, and, irreducible over $\F$. Let $H^S$ and $H^T$ be the Hermitian form induced by $S$ and $T$ respectively.
\begin{enumerate}
 \item[(i)]Then $S$ and $T$ are conjugate in ${\rm I}(\V, B)$ if and only if $H^S$ and $H^T$ are equivalent. 

\item[(ii)] Let $Z(T)$ be the centralizer of $T$ in ${\rm I}(\V, B)$. Then an isometry $C$ is in $Z(T)$ if and only if $C$ preserves $H^T$, i.e. $Z(T)= U(\V^T, H^T)$.\end{enumerate}
\end{theorem}
\begin{proof}
Suppose $S$ is conjugate to $T$ in ${\rm I}(\V, B)$. Let $C$ in ${\rm I}(\V, B)$ be such that $T=CSC^{-1}$. Then  $C:\V^S \to \V^T$ is an
$\F$-isomorphism. For $l \geq 1$, and $v$ in $\V^S$,
$$
C(s^l v)= C \circ S^l (v)= T^l\circ C(v)=t^lC(v)=f(s^l)C(v).$$
It follows that, for all $e$ in $\E^S$, and $v$ in $\V^S$,   $C(e v)=f(e)C(v)$. 
 For $u, v$ in $\V^S$, note that
\begin{eqnarray*}
h^S(f^{-1}(H^T(C(u), C(v))) &=& h^S \circ f^{-1}(H^T(C(u), C(v)))\\
&=& h^T(H^T(C(u),C(v))\\ 
& =& B(C(u), C(v)))\\
&=& B(u, v)=h^S(H^S(u, v)).
\end{eqnarray*}
Hence, by the uniqueness of $H^S$ we have,   $f^{-1}(H^T(C(u), C(v)))=H^S(u, v)$,  i.e.  $H^T(C(u), C(v))=f(H^S(u, v))$.
This shows that $H^S$ and $H^T$ are equivalent.

Conversely, suppose $H^S$ and $H^T$ are equivalent. Let $C: \V^S \to \V^T$ be an $\F$-isomorphism such that $(i)$ and $(ii)$ in \defref{here} hold. We have for $v$ in $\V$,
\begin{eqnarray*}
CS(v) &=& C(sv)\\
&=& f(s)C(v)\\
&=& tC(v)=T C(v).
\end{eqnarray*}
that is, $CSC^{-1}=T$. Further, for $x$, $y$ in $\V$, 
\begin{eqnarray*}
 B(C(x), C(y))&=&h^T(H^T(C(x), C(y)))\\
&=& h^T(f(H^S(x, y)))\\
& =& h^S(H^S(x, y))=B(x, y).
\end{eqnarray*}
Hence $C:\V \to \V$ is an isometry. This completes the proof of $(i)$. 

$(ii)$ Note that an invertible linear transformation $C:\V \to \V$ is $\E^T$-linear if and only if $CT=TC$.  Now replacing $S$ by $T$, and $f$ by identity in the proof of (i) the theorem follows.
\end{proof}

\subsection{Conjugacy classes}\label{cc}
\begin{theorem}\label{ppp}
Let $T$ be an element of ${\rm I}(\V, B)$. Let the minimal polynomial of $T$ be $p(x)^d$,  where $p(x)=x-1$ or $p(x)$ is monic, self-dual and irreducible over $\F$.
\begin{enumerate}
\item{
The conjugacy class of $T$ in ${\rm I}(\V, B)$ is determined by the following data.
\begin{itemize}
\item[(i)]{ The elementary divisors of $T$. }
\item[(ii)]{ The finite sequence of equivalence classes of Hermitian spaces 
$$\{(\V^T_{d_1}, H^T_{d_1}),\cdots,(\V^T_{d_k}, H^T_{d_k})\},$$
where $1 \leq d_1 < d_2 < \cdots< d_k=d$, and for each $i$, $H_{d_i}^T$ takes values in the cyclic algebra $\E_{d_i}=\F[x]/(p(x)^{d_i})$.}
\end{itemize}
}
\item{ The centralizer of $T$ is the direct product 
$$U(\V^T_{d_1}, H^T_{d_1}) \times \cdots\times U(\V^T_{d_k}, H^T_{d_k}).$$}

\end{enumerate}

\end{theorem}
\begin{proof}
 Suppose $S: \V \to \V$ and $T:\V \to \V$ are two isometries. If $S$ and $T$ are conjugate in ${\rm I}(\V, B)$, then by the structure theory of linear operators and \thmref{lppp},  it is clear that they have the same data.

Conversely, suppose $S$ and $T$ have the same data.
 The elementary divisors of $S$ and $T$ determine orthogonal decompositions of $\V$ as
\begin{equation} \label{e1} \V=\V^S_{d_1} \oplus\cdots\oplus \V^S_{d_k},\end{equation}
\begin{equation} \label{e2} \V=\V^T_{d_1} \oplus\cdots\oplus \V^T_{d_k},\end{equation}
where $1 \leq d_1 < \cdots< d_k=d$, and for each $i$, $\V^S_{d_i}$, resp. $\V^T_{d_i}$ is free when considered as a module over $\E^S_{d_i}$, resp. $\E^T_{d_i}$. Since $\E^S_{d_i}$ and $\E^T_{d_i}$ are isomorphic, wthout loss of generality, we identify them with $\E_{d_i}=
 \F[x]/(p(x)^{d_i})$. 
Since $S$ and $T$ have the same set of elementary divisors, $\V^S_{d_i}$ is isomorphic to $\V^T_{d_i}$ as a free module over $\E_{d_i}$, for $i=1,2,...,k$.
 For each $i=1,2,\cdots,k$, since $(\V^S_{d_i}, H^S_{d_i})$ is equivalent to $(\V^T_{d_i}, H^T_{d_i})$, by \thmref{lppp}, $S|_{\V^S_{d_i}}$ is conjugate to $T|_{\V^T_{d_i}}$. Hence $S$ is conjugate to $T$.   

The description of $Z(T)$ is clear from the orthogonal decomposition of $\V$  and  part (2) of \thmref{lppp}.
\end{proof}

\subsection{The $z$-classes}$\;$

\begin{theorem}\label{zc1}
Let $T: \V \to \V$ be an element in ${\rm I}(\V, B)$ such that  $m_T(x)=p(x)^d$, where $p(x)$ is self-dual and irreducible over $\F$. The $z$-class of $T$ is determined by the following data. 
\begin{itemize}
\item[(i)]{ The degree  $m$ of $p(x)$. }
\item[(ii)]{ A non-decreasing sequence of integers $(d_1,\ldots,d_k)$ which corresponds to the secondary partition  $\pi: \frac{n}{m}=\Sigma_{i=1}^k d_il_i$.  
}
\item[(iii)]{ A sequence $(\E_{d_1},\ldots ,\E_{d_k})$ of isomorphism classes of cyclic algebras over $\F$, where for each $i=1,2,\ldots,k$, $\E_{d_i}$ is isomorphic to $\F[x]/(p(x)^{d_i})$. }
\item[(iv)]{ A finite sequence of equivalence classes of Hermitian forms
\\ \hbox{$(H_{d_1},\ldots ,H_{d_k})$}, where each $H_{d_i}$ takes values in $\E_{d_i}$.}\end{itemize}
\end{theorem}
\begin{proof}
Let $S$ and $T$ be two isometries of $(\V, B)$ with same data $(i)-(iv)$.  We use the same notations as in the previous theorem. Let $m_S(x)=p(x)^d$, and $m_T(x)=q(x)^d$, degree of $p(x)=$ degree of $q(x)=m$. For each $i=1, \ldots, k$,  $\E^S_{d_i}$ and $\E^T_{d_i}$ are isomorphic. Let $f_i: \E_{d_i}^S \to \E_{d_i}^T$ be one such isomorphism. 

  For simplicity, for each $i$, we identify $\E_{d_i}^S$, and $\E^T_{d_i}$ with $\E_{d_i}$. Moreover,  following \remref{r1} assume $h^S=h^T$.  

  Since the Hermitian forms $H_{d_i}^S$ and $H_{d_i}^T$ are equivalent, let  $F_i: (\V_{d_i}^S, H_{d_i}^S) \to (\V_{d_i}^T, H_{d_i}^T)$ be an equivalence of the Hermitian spaces. We see that, for $u, v \in \V_{d_i}^S$, 
\begin{eqnarray*}
B(F_i u, F_i v)&=&h^T (H_{d_i}^T(F_i u, F_i v))\\&=&h^T \circ f_i (H_{d_i}^S (u, v))\\&=&h^S(H_{d_i}^S(u, v)), \hbox{ see, \remref{r1}}, \\&=& B(u, v).\end{eqnarray*}
Thus  $F_i$ is an isometry with respect to $B$. Further, $F_i$ conjugates 
$Z(S|_{\V_{d_i}^S})=U(\V_{d_i}^S, H_{d_i}^S)$ and  
$Z(T|_{\V_{d_i}^T})=U(\V_{d_i}^T, H_{d_i}^T)$.  Thus, $F=F_1 \oplus F_2 \oplus \cdots \oplus F_k$ is an isometry of $(\V,  B)$ and $F$  conjugates $Z(S)$ and $Z(T)$. 

Conversely,  suppose $S$ and $T$ are in the same $z$-class. Replacing $S$ by its conjugate, we may assume,  $Z(S)=Z(T)$. Hence by \hbox{part (2) of \thmref{cc} } we see that $S$ and $T$ have isomorphic decompositions \eqnref{e1} and \eqnref{e2}. After renaming the indices, if necessary, we may assume further that for $i=1,2,...,k$,  $(\V^S_{d_i}, H^S_{d_i})$ and $(\V^T_{d_i}, H^T_{d_i})$ are equivalent. In particular, $\E^S_{d_i}$ and $\E^T_{d_i}$ are isomorphic,  and their common dimension over $\F$ is $md_i$. This implies $\deg m_S(x)=\deg m_T(x)$.   Consequently we attach the partition (see, \remref{r0}) $\pi: \frac{n}{m}=\Sigma_{i=1}^k d_i l_i$ to the $z$-class  and it follows that $S$ and $T$ have the same data $(i)-(iv)$. 

This completes the proof. 
\end{proof}

\subsection{The minimal polynomial is $(x+1)^d$}$\;$

Note that,  two isometries $S$ and $T$ are conjugate if and only if $-S$ and $-T$ are conjugate. Now,  suppose $T$ is an isometry with minimal polynomial $(x-1)^d$. Then $-T: \V \to \V$ is also an isometry, and $m_{-T}(x)=(x+1)^d$. Conversely, if $T$ is unipotent, then $-T$ has minimal polynomial $(x+1)^d$. Thus this case is reduced to the unipotent case, and the parametrization of the conjugacy and the $z$-classes of $T$ are similar to that of $-T$.  

\subsection{The minimal polynomial is a product of pairwise dual polynomials} \; \; 

\medskip 
 Suppose $T:\V \to \V$ is an element in ${\rm I}(\V, B)$ such that $m_T(x)= q(x)^d q^{\ast}(x)^d$, where $q(x)$, $q^{\ast}(x)$ are irreducible polynomials over $\F$ of degree $m$ and are dual to each-other. For our purpose, it is enough to consider the case when $T$ is semisimple. So assume, $d=1$. Let $\V_q=\hbox{ ker } q(T)$, $\V_{q^{\ast}}=\hbox{ ker } q^{\ast}(T)$. We have
\begin{equation}\label{vd1}\V=\V_q + \V_{q^{\ast}},\end{equation}
and $B|_{\V_{q}}=0=B|_{\V_{q^{\ast}}}$, dim $\V_q=\hbox{ dim } \V_{q^{\ast}}$. Since $B$ is non-degenerate, we can choose a basis $\{e_1,....,e_m, f_1,...,f_m\}$ of $\V$ such that for each $i$, $e_i \in \V_q$, $f_i \in \V_{q_{\ast}}$, and for all $i, j=1, \ldots m$, 
$$B(e_i, e_i)=0=B(f_i, f_i), \; B(e_i, f_j)=\delta_{ij} \hbox{ or }-\delta_{ij}.$$
For each $w^{\ast} \in \V_{q^{\ast}}$, define the linear map
$w^{\ast}: v \to B(v, w)$. These maps enable us to identify $\V_{q^{\ast}}$ with the dual of $\V_q$. Thus $(\V,B)$ is a \emph{standard space}, see, \cite[Section-2.1]{gk2}, and $T=T_L + T_L^{\ast}$, where $T_L$, the restriction of $T$ to $\V_q$, is an element of ${\rm GL}(\V_q)$. Conversely, given an element in ${\rm GL}(\V_q)$,  it can be extended to an isometry of $(\V, B)$. Hence the conjugacy classes in $\I(\V, B)$ are parametrized by the usual structure theory of linear maps. 

\medskip Define a form $H_T$ on $\V$ as  follows: For $u, v \in \V$, $H_T(u, v)=B(Tu, v)$. Clearly, if $S$ in $I(\V, B)$ commutes with $T$, then
$$H_T(Su, Sv)=B(TSu, Sv)=B(STu, Sv)=B(Tu, v)=H_T(u, v).$$
Conversely, if $S$ preserves $H_T$, then $H_T(Su, Sv)=H_T(u, v)$ implies that for $u, v \in \V$, 
$B(STu, Sv)=B(TSu, Sv)$. By the non-degeneracy of $B$, it follows that $S$ commutes with $T$.

Now, suppose $\E$  is the splitting field of $q(x)$ (hence of $q^{\ast}(x)$ also).  Let $\alpha_1, \cdots, \alpha_k$ be distinct roots of $q(x)$  in $\E$. There is a unique automorphism $e \mapsto \bar e$ which maps $\alpha_i \to \alpha_i^{-1}$. Further $\V$ over $\E$ has a decomposition into eigenspaces: 
$$\V=\bigoplus_{i=1}^k \  ( \V_{\alpha_i} + \V_{\alpha_i^{-1}}). $$
Without loss of generality, assume $\V=\V_{\alpha} + \V_{\alpha^{-1}}$. 
Then $H_T$ defines an $\E$-valued Hermitian form on $\V$: when $u \in \V_{\alpha}$ and $v \in \V_{\alpha^{-1}}$, we have $H_T(u, v)=\overline{ H_T(v, u)}$. Thus,  $Z(T)={\rm U}(\V, H_T)$.  
We have now proved the following lemma. 
\begin{lemma}\label{zc2}
Let $\dim \V$ be even. Let  $T$ be a semisimple element in ${\rm I}(\V, B)$ such that $m_T(x)=q(x)q^{\ast}(x)$, where $q(x)$, $q^{\ast}(x)$ are irreducible polynomials over $\F$ and they are dual to each-other. Let $\E$ be the splitting field of the minimal polynomial of $T$. Then the $z$-class of $T$ is determined by 
\begin{itemize}
\item[(i)]{The degree of $q(x)$. }
\item[(ii)]{Equivalence class of $\E$-valued Hermitian forms $H_T$ on $\V$. }
\end{itemize}
\end{lemma}

\section{Classification over a perfect field}\label{perfect}
Let $\F$ be a perfect field of characteristic different from two. The group of isometries ${\rm I}(\V, B)$ consists of  $\F$-points of a linear algebraic group defined over $ \F$. Thus each $T$ in ${\rm I}(\V, B)$ has the Jordan decomposition $T=T_s T_u$, where $T_s$ is semisimple (that is, every $T_s$-invariant subspace has a $T_s$-invariant complement) and $T_u$ is unipotent. Moreover $T_s$, $T_u$ are also elements of ${\rm I}(\V, B)$, they commute with each other, and they are polynomials in $T$ (see, \cite[Chapter 15]{hum}). Moreover we have $Z(T)=Z(T_s)\cap Z(T_u)$. To some extent, the Jordan decomposition reduces the study of conjugacy and $z$-classes in ${\rm I}(\V, B)$ to the study of conjugacy and $z$-classes of semisimple and unipotent elements.
 Suppose $T: \V \to \V$ is a semisimple isometry with prime and self-dual minimal polynomial.  Suppose $\E=\F[x]/(p(x))$.  Then $\E$ is a finite simple field extension of $\F$, $[\E: \F]=$ degree of $p(x)$. Thus the underlying cyclic algebras in \thmref{zc1} (or, \thmref{zc2}) are isomorphic to the field $\E$, and the Hermitian forms $H_{d_i}$ are $\E$-valued.

Now suppose $T$ is an arbitrary semisimple isometry, and let its minimal polynomial be a product of pairwise distinct prime polynomials over $\F$. Let 
 $$m_T(x)=(x-1)^e(x+1)^{f} ~ \Pi_{i=1}^k p_i(x)\;\Pi_{j=1}^l q_j(x) q_j^{\ast}(x),$$
  where $e, f=0$ or $1$, $p_1(x),...,p_k(x)$ are self-dual, and for $j=1,2,...,l$, $q_j(x)$ is dual to $q_j^{\ast}(x)$. Suppose for each $i$, the degree of $p_i(x)$ is $2m_i$, and for each $j$, degree of $q_j(x)$ is $l_j$. Let the characteristic polynomial of $T$ be 
  $$\chi_T(x)=(x-1)^l(x+1)^m ~ \Pi_{i=1}^k p_i(x)^{d_i}\;\Pi_{j=1}^l q_j(x)^{e_j}q^{\ast}_j(x)^{e_j}.$$
The primary decomposition of $\V$ with respect to $T$ is determined by the minimal and the characteristic polynomial of $T$.  We get the following orthogonal decomposition of $\V$ into $T$-invariant subspaces:
 \begin{equation}\label{pe1}\V=\V_1 \oplus \V_{-1}\oplus_{i=1}^{k} \V_i \bigoplus \oplus_{j=1}^{l} (\W_j + \W^{\ast}_j), \end{equation}
where $\V_1=\hbox{ker }(T-I)^l$, $\V_{-1}=\hbox{ker }(T+I)^m$, for each $i=1,...,k$, \hbox{$\V_i=\hbox{ker }p_i(T)$}, and for each $j=1,2,...,l$, $\W_j=\hbox{ker }q_j(T)$, $\W_j^{\ast}=\hbox{ker }q_j^{\ast}(T)$. We have, dim $\V_i=2m_id_i$, and dim $\W_j=l_j e_j$. Let $\E_i$ be the field isomorphic to $\F[x]/(p_i(x))$, and let $\K_j$ be the field isomorphic to $\F[x]/(q_j(x))$. As a vector space over $\E_i$, $\V_i$ is the direct sum of $d_i$ copies of $\E_i$.  The $z$-class of $T$ is determined by the $z$-classes of the restrictions of $T$  to each component in the primary decomposition \eqnref{pe1}. Note that $T|_{\V_1}=I$, $T|_{\V_{-1}}=-I$. Since $I$ and $-I$ belong to the center of the group,  the $z$-class of $T$ restricted to $\V_1$ or $\V_{-1}$ is determined by $\dim \V_1$ or $\dim \V_{-1}$. Now, the following theorem follows from \thmref{zc1} and \lemref{zc2}.

\begin{theorem}\label{zcs}
Suppose $\F$ is perfect. Let $T: \V \to \V$ be a semisimple element in ${\rm I}(\V, B)$.  The $z$-class of $T$ is determined by
\begin{itemize}
\item[(i)]{A finite sequence of integers $(l, m; m_1,...,m_{k_1}; l_1,...,l_{k_2})$. }
\item[(ii)]{  A partition of $n$, $\pi: n=l+m+2 \Sigma_{i=1}^{k_1} m_i d_i + 2 \Sigma_{j=1}^{k_2} l_j e_j$.}
\item[(iii)]{ Field extensions $\E_i$, $1 \leq i \leq k_1$ of $\F$, $[\E_i: \F]=2m_i$, and $\K_j$, $1 \leq j \leq k_2$, $[\K_j: \F]= l_j$. }
\item[(iv)]{ Equivalence classes of   $\E_i$-valued Hermitian forms $H_i$, $1 \leq i \leq k_1$, and $\K_j$-valued Hermitian forms  $H'_j$, $1 \leq j \leq k_2$.  }\end{itemize}
\end{theorem}

\section{Finiteness of the $z$-classes: Proof of \thmref{zc}} 
If there are only finitely many $z$-classes of semisimple and unipotent elements, it follows from the Jordan decomposition that there are only finitely many $z$-classes. So it is enough to show the finiteness of $z$-classes of semisimple and unipotent elements respectively. 

Suppose $\F$ is a perfect field that has only finitely many field extensions of degree at most $n$. Then the number of distinct equivalence classes of quadratic forms of rank at most $n$ is finite. Hence the number of equivalence classes of Hermitian forms of rank at most $n$ over an extension field of $\F$ is finite.  Combining this fact with \thmref{zcs}, and the fact that there are only finitely many partitions of $n$, we  have that there are only finitely many $z$-classes of semisimple elements. 

Similarly,   it follows from \thmref{ppp} that there are only finitely many conjugacy classes of unipotent elements;  hence there are only finitely many $z$-classes of unipotent elements. 

This completes the proof of \thmref{zc}.

\end{document}